\documentclass[reqno,oneside]{amsart}
\usepackage{hyperref}
\usepackage{geometry}

\usepackage{anyfontsize}
\usepackage{graphicx}
\usepackage{amsmath}
\usepackage{amsthm}
\usepackage{amssymb, color}%
\usepackage[numbers, square]{natbib}
\usepackage{mathrsfs}
\usepackage{bbm}
\usepackage{tikz}
\usepackage{enumitem}
\usepackage{mathptmx}
\usepackage{enumitem}

\theoremstyle{plain}

\newtheorem{theorem}{Theorem}[section]
\newtheorem{cor}[theorem]{Corollary}
\newtheorem{lemma}[theorem]{Lemma}
\newtheorem{prop}[theorem]{Proposition}

\theoremstyle{remark}
\newtheorem{definition}[theorem]{Definition}
\newtheorem{rem}[theorem]{Remark}
\newtheorem{example}{Example}

\newcommand{\diam}{{\rm diam}}
\newcommand{\N}{\mathbb{N}}
\newcommand{\R}{\mathbb{R}}
\newcommand{\Z}{\mathbb{Z}}
\newcommand{\E}{\mathbb{E}}
\newcommand{\mmp}{\mathbb{P}}

\DeclareMathOperator{\ind}{\mathbbm{1}}

\newcommand{\tovague}{\overset{{\rm v}}{\underset{d\to\infty}\longrightarrow}}

\newcommand{\conv}{\mathop{\mathrm{conv}}}
\newcommand{\oconv}{\mathop{\overline{\mathrm{conv}}}}

\begin{document}
\title[Random Walks in the High-Dimensional Limit]{Random Walks in the High-Dimensional Limit II:\\The Crinkled Subordinator}

\author{Zakhar Kabluchko}
\address{Zakhar Kabluchko: Institut f\"ur Mathematische Stochastik,
Westf\"alische Wilhelms-Universit\"at M\"unster, M\"unster, Germany
}
\email{zakhar.kabluchko@uni-muenster.de}

\author{Alexander Marynych}
\address{Alexander Marynych: Faculty of Computer Science and Cybernetics, Taras Shevchenko National University of Kyiv}
\email{marynych@unicyb.kiev.ua}

\author{Kilian Raschel}
\address{Kilian Raschel:
Laboratoire Angevin de Recherche en Math\'{e}matiques, Universit\'{e} d’Angers, CNRS, Angers, France
}
\email{raschel@math.cnrs.fr}

\begin{abstract}
A crinkled subordinator is an $\ell^2$-valued random process which can be thought of as a version of the usual one-dimensional subordinator with each out of countably many jumps being in a direction orthogonal to the directions of all other jumps. We show that the path of a $d$-dimensional random walk with $n$ independent identically distributed steps with heavy-tailed distribution of the radial components and asymptotically orthogonal angular components
converges in distribution in the Hausdorff distance up to isometry and also in the Gromov--Hausdorff sense, if viewed as a random metric space, to the closed range of a crinkled subordinator, as $d,n\to\infty$.
\end{abstract}




\maketitle

\section{Introduction}\label{subsec:introduction}
Let $\ell^2$ be the infinite-dimensional (real)\ Hilbert space of square-summable sequences endowed with the standard Hilbert norm
\begin{equation*}
   \|(x_1,x_2,\ldots)\|_{2}:=\sqrt{\sum_{k=1}^{\infty} x_k^2},\quad  (x_1,x_2,\ldots)\in\ell^2,
\end{equation*}
and the standard inner product
\begin{equation*}
   \langle (x_1,x_2,\ldots),(y_1,y_2,\ldots)\rangle_2:=\sum_{k=1}^{\infty} x_k y_k,\quad (x_1,x_2,\ldots),\quad (y_1,y_2,\ldots)\in\ell^2.
\end{equation*}
Fix the standard  orthonormal basis $(e_k)_{k\in\N}$ of $\ell^2$ and consider the natural embeddings
\begin{equation*}
   \R\subset \R^2\subset\cdots\subset\R^d \subset\cdots \subset\ell^2,
\end{equation*}
obtained by identifying $\R^d$ with the linear span of
$(e_1,e_2,\ldots,e_d)$, $d\in\N$. This will allow us throughout the paper to treat elements of $\R^d$ as elements of $\ell^2$ and use the notation $\|x\|_2$ (respectively, $\langle x,y\rangle_2$) for the usual Euclidean norm of $x\in\R^d$ (respectively, the standard inner product of $x,y\in \R^d$). Denote by $\rho_2(x,y):=\|x-y\|_2$ the metric on $\ell^2$ induced by the norm $\|\cdot\|_2$.

Let $(X^{(d)})_{d\in\N}$ be a sequence of random variables defined on a common probability space $(\Omega,\mathcal{F},\mathbb{P})$ such that $X^{(d)}$ takes values in $\R^d$ (identified with the linear span of
$(e_1,e_2,\ldots,e_d)$ in $\ell^2$), $d\in\N$. Assume that the space $(\Omega,\mathcal{F},\mathbb{P})$ is reach enough to accommodate a sequence $(X_i^{(d)})_{i\in\N}$ of independent  copies of $X^{(d)}$, for each $d\in\N$. Consider a family of random walks defined via
\begin{equation}\label{eq:rw_def}
S^{(d)}_0:=0,\quad S^{(d)}_k:=X^{(d)}_1+X^{(d)}_2+\cdots+X^{(d)}_k,\quad k\in\N,
\end{equation}
for each $d\in\N$. Let $n=n(d)$ be an arbitrary sequence of positive integers such that $n(d)\to \infty$ as $d\to\infty$. By default, the notation $d\to\infty$  implies that also $n= n(d)\to\infty$. Denote by $\widehat{S}^{(d)}_n$ the piecewise-linear interpolation obtained by joining the consecutive points $S^{(d)}_0,S^{(d)}_1,\ldots,S^{(d)}_n$ by line segments. By construction, every $\widehat{S}^{(d)}_n$ can be regarded as a continuous piecewise-linear curve in $\ell^2$ starting at the origin and living in the finite-dimensional subspace $\R^d$, $d\in\N$.

This paper is a continuation of~\cite{KabMar:2022} and devoted to finding an answer to the question: How does the curve $\widehat{S}^{(d)}_n$ (after an appropriate renormalization and up to isometries in $\ell^2$) look like when $d$ and, therefore, $n$ tend to infinity? Does it approach some deterministic or random curve in $\ell^2$? Under the assumptions $\E X^{(d)}=0$, $\E \|X^{(d)}\|^2=1$ and the components of $X^{(d)}$ are uncorrelated (plus some mild technical assumptions), it was proved in~\cite{KabMar:2022} that $(\widehat{S}^{(d)}_n/\sqrt{n},\rho_2)$, regarded as a compact metric space, converges in probability in the Hausdorff distance up to isometry and also in the Gromov--Hausdorff sense, see Section~\ref{sec:GH_convergence} below for the definitions, to a deterministic metric space called the Wiener spiral. An isometric copy of the Wiener Spiral in the space $\ell^2$ is given by a continuous curve $\mathbb{W}:=(\widetilde{w}_t)_{t\in [0,1]}$, where
\begin{equation}\label{eq:wiener_spiral_l2}
\widetilde{w}_t:=\frac{2\sqrt{2}}{\pi}\sum_{k=1}^{\infty}\frac{\sin(\pi(k-1/2)t)}{2k-1}e_k,\quad t\in [0,1],
\end{equation}
which possesses peculiar properties
\begin{equation}\label{eq:wiener_metric}
\|\widetilde{w}_t-\widetilde{w}_s\|_2=\sqrt{|t-s|},\quad 0\leq s,t\leq 1,
\end{equation}
and also
\begin{equation}\label{eq:wiener_orthogonal}
\langle \widetilde{w}_t-\widetilde{w}_s,\widetilde{w}_u-\widetilde{w}_v\rangle_2=0,\quad 0\leq v\leq u\leq s\leq t\leq 1.
\end{equation}
Note that~\eqref{eq:wiener_metric} means that, as a metric space, $(\mathbb{W},\rho_2)$ is isometric to the interval $[0,1]$ endowed with the distance $\sqrt{|t-s|}$. The Wiener spiral is also isometric to a continuous curve $(w_t)_{t\in [0,1]}$ in $L_2([0,1])$, given by $w_t=\ind_{[0,t]}(\cdot)\in L_2([0,1])$, $t\in [0,1]$.  This can be seen by noting that
\begin{equation*}
   \langle w_t,w_s\rangle_{L_2([0,1])}=\min(t,s)=\langle \widetilde{w}_t,\widetilde{w}_s\rangle_{2}.
\end{equation*}
It is worth mentioning that replacing $e_k$ in~\eqref{eq:wiener_spiral_l2} by ${\rm N}_k$, where $({\rm N}_k)_{k\in\N}$ are independent  identically  distributed (i.i.d.)\ standard normal random variables, gives the Karhunen--Lo\'{e}ve expansion of a  standard Brownian motion.

The aforementioned result of~\cite{KabMar:2022} can be compared with the classical functional weak law of large numbers for one-dimensional random walks. Recall that the latter tells us that finiteness of the first moment of a generic step implies uniform convergence of the path of the rescaled random walk to a deterministic linear function in probability; see \cite{Glynn+Whitt:1988} for example. In the setting of~\cite{KabMar:2022}, the authors show that the finiteness of $\E\|X^{(d)}\|^2=1$ also implies convergence to a deterministic limit. A natural question arising from this comparison is the following.
It is known that if a generic step of a one-dimensional random walk is a.s.\ positive, has infinite mean and its distribution has a regularly varying at infinity tail, then the path of the rescaled random walk converges to a random limit being the path of a subordinator. Thus, suppose now that the distribution of $\|X^{(d)}\|^2$ is regularly varying at infinity. Keeping in mind the above analogy with one-dimensional random walks, it is natural to expect that in this scenario, $\widehat{S}^{(d)}_n$ converges as $d\to\infty$, after an appropriate rescaling, to a genuinely random curve in $\ell^2$, a path of a certain $\ell^2$-valued random process derived from a subordinator. The main result of our paper confirms these expectations.

\section{Assumptions, definitions and main results}

\subsection{Assumptions}\label{sec:assumptions}

We shall now present our assumptions on the distributions of $(X^{(d)})_{d\in\N}$ which will be used throughout the paper. The components of the vectors $X_i^{(d)}$ (independent copies of $X^{(d)}$), and $S_{i}^{(d)}$ are denoted by $X_i^{(d)} = (X_{i,1}^{(d)},\dots,X_{i,d}^{(d)})$ and $S_i^{(d)} = (S_{i,1}^{(d)},\dots, S_{i,d}^{(d)})$, respectively. Furthermore, let
\begin{equation*}
   \Theta^{(d)}:=\frac{X^{(d)}}{\|X^{(d)}\|_2},\quad\Theta^{(d)}_i:=\frac{X^{(d)}_i}{\|X^{(d)}_i\|_2},\quad i\in\N,\quad d\in\N,
\end{equation*}
denote the angular components of $X^{(d)}$ and $X^{(d)}_i$'s.

Suppose that the following hypotheses hold:
\begin{enumerate}[label={\rm(\alph{*})},ref=(\alph{*})]
\item\label{it(a)}There exist constants $(a(k))_{k\in\mathbb{N}}$ and a L\'{e}vy measure $\nu$ on $(0,\infty]$ satisfying
\begin{equation}\label{eq:nu_integrable}
\int_{(0,\infty)}\min(1,x)\nu({\rm d}x)<\infty,\quad \nu({\infty})=0,
\end{equation}
and such that
\begin{equation}\label{eq:regular_variation_gen}
n\mmp\{(a(n))^{-1}\|X^{(d)}\|_2^2\in\cdot\}\tovague \nu(\cdot),
\end{equation}
where $\tovague$ stands for the vague convergence of measures on $(0,\infty]$. Suppose, further, that
\begin{equation}\label{eq:regular_variation_trunc_sec_moment}
\lim_{s\to 0+}\limsup_{d\to\infty}\frac{n}{a(n)}\mathbb{E}\left(\|X^{(d)}\|_2^2\ind_{\{\|X^{(d)}\|_2^2\leq s a(n)\}}\right)=0.
\end{equation}

\item\label{it(b)}With the sequence $(a(k))_{k\in\mathbb{N}}$ defined in part~\ref{it(a)}, the following relation holds true, for all fixed $s>0$ and $\varepsilon>0$:
\begin{equation}\label{eq:assump_orthogonal}
\lim_{d\to\infty}\mmp\left\{\left|\left\langle \Theta^{(d)}_1,\Theta^{(d)}_2\right\rangle_2\right|>\varepsilon\;\Big|\; \|X_1^{(d)}\|_2^2\geq s a(n),\|X_2^{(d)}\|_2^2\geq s a(n)\right\}=0.
\end{equation}

\item\label{it(c)}With the sequence $(a(k))_{k\in\mathbb{N}}$ defined in part~\ref{it(a)}:
\begin{equation}\label{eq:centering_is_negligible}
\lim_{s\to 0+}\limsup_{d\to\infty}\frac{n}{\sqrt{a(n)}}\|\E X^{(d)}\ind_{\{\|X^{(d)}\|_2^2 \leq s a(n)\}}\|_2=0.
\end{equation}
\end{enumerate}
As we shall now see, there are several cases in which the assumptions~\ref{it(a)}--\ref{it(c)} can be significantly simplified.

\vspace{2mm}
\noindent
{\sc Identically distributed $\|X^{(d)}\|^2$, $d\in\mathbb{N}$.} Under the assumption that the distribution of $\|X^{(d)}\|^2$ does not depend on $d$, condition~\eqref{eq:regular_variation_gen} is equivalent to the following one. There exist $\alpha\in (0,1)$ and a function $L$ slowly varying at infinity such that, for all $d\in\N$,
\begin{equation}\label{eq:regular_variation1}
\mmp\{\|X^{(d)}\|^2>t\}=t^{-\alpha}L(t),\quad t>0.
\end{equation}
In this case $\nu(x,\infty)=x^{-\alpha}$ for all $x>0$, and $(a(k))_{k\in\mathbb{N}}$ can be any positive sequence satisfying
\begin{equation*}
   \lim_{n\to\infty}n\mmp\{\|X^{(d)}\|^2>a(n)\}=\lim_{n\to\infty}\frac{nL(a(n))}{a(n)^{\alpha}}=1.
\end{equation*}
The existence of such a sequence follows by a standard argument involving de Bruijn conjugates; see~Chapter~1.5.7 in \cite{Bingham+Goldie+Teugels:1989}. Note that the restriction $\alpha\in (0,1)$ comes from the fact that we require the L\'{e}vy measure to satisfy the integrability condition~\eqref{eq:nu_integrable}. Furthermore, in this case the condition~\eqref{eq:regular_variation_trunc_sec_moment} holds automatically; see~\eqref{eq:regular_variation_trunc_sec_moment_verify} below.

\vspace{2mm}
\noindent
{\sc Independent radial and angular components of $X^{(d)}$.} The conditions in parts~\ref{it(b)} and \ref{it(c)} take a particularly  simple form if the radial and angular components of $X^{(d)}$ are independent. More precisely, if $\|X^{(d)}\|_2$ and $\Theta^{(d)}$ are independent, then~\eqref{eq:assump_orthogonal} is equivalent to
\begin{equation}\label{eq:assump_orthogonal_indep}
\left\langle \Theta_1^{(d)},\Theta_2^{(d)}\right\rangle_2 \overset{\mmp}{\longrightarrow} 0,\quad d\to\infty.
\end{equation}
That is to say, the condition of part~\ref{it(b)} is equivalent to saying that two independent copies of $X^{(d)}$ are asymptotically orthogonal in probability, as $d\to\infty$. If also $\E \Theta^{(d)}=0$, then~\eqref{eq:centering_is_negligible} holds automatically, since the truncated expectation in~\eqref{eq:centering_is_negligible} is equal to zero for independent $\|X^{(d)}\|_2$ and $\Theta^{(d)}$.

\vspace{2mm}
\noindent
{\sc Symmetric distribution of $X^{(d)}$.} Assume that the law of $X^{(d)}$ is the same as that of $-X^{(d)}$, for all $d\in\N$. In this case the truncated expectation in~\eqref{eq:centering_is_negligible} is equal to zero, since the function of $X^{(d)}$ under $\E$ in~\eqref{eq:centering_is_negligible} is odd. Thus, \ref{it(c)} holds automatically.

\subsection{The crinkled subordinator}

Fix $T>0$. It is known, see Theorem 7.1 on p.~214 in~\cite{Resnick:2007}, that the assumptions in part~\ref{it(a)} imply
\begin{equation}\label{eq:FLT}
\left(\frac{\|X_1^{(d)}\|^2+\|X_2^{(d)}\|^2+\cdots+\|X_{\lfloor nt\rfloor}^{(d)}\|^2}{a(n)}\right)_{t\in [0,T]}\Longrightarrow (\mathcal{S}_{\nu}(t))_{t\in [0,T]},\quad n\to\infty,
\end{equation}
in the Skorokhdod space of c\`{a}dl\`{a}g functions defined on $[0,T]$ and endowed with the Skorokhod $J_1$-topology. Here, $\mathcal{S}_{\nu}$ is a subordinator, whose construction, which we are now going to recall, is of major importance for everything to follow.

A subordinator is an increasing stochastic process that has independent and homogeneous increments.   For the purposes of the present paper, the following definition (called It\^{o}'s decomposition)\ serves best. Let $\mathcal{P}:=\sum_{k}\delta_{(x_k,y_k)}$ be a Poisson random measure on $[0,\infty)\times (0,\infty]$ with the intensity measure $\mathbb{LEB}\times \nu$, where $\mathbb{LEB}$ denotes Lebesgue measure and $\nu$ is the L\'evy measure as in~\ref{it(a)}. Here and in what follows, $\delta_x$ denotes a Dirac measure at $x$. The stochastic process
\begin{equation*}
   \mathcal{S}_{\nu}(t):=\sum_{k:\;x_k\leq t}y_k=\int_{[0,t]\times (0,\infty]}y\mathcal{P}({\rm d}x,{\rm d}y),\quad t\geq 0,
\end{equation*}
is called a drift-free subordinator with L\'{e}vy measure $\nu$. Condition~\eqref{eq:nu_integrable} ensures that the sum above is a.s.~finite for all $t\geq 0$. In case when $\nu(x,\infty)=x^{-\alpha}$, $x>0$, for some $\alpha \in (0,1)$, the subordinator $\mathcal{S}_{\nu}$ is called $\alpha$-stable. We allow $\nu$ to be a finite measure, in which case $\mathcal{S}_{\nu}$ is a compound Poisson process.

\begin{definition}
A crinkled subordinator with L\'{e}vy measure $\nu$ (and with respect to the chosen basis of $\ell^2$) is an $\ell^2$-valued stochastic process $(\mathcal{C}_{\nu}(t))_{t\geq 0}$ defined by
\begin{equation*}
   \mathcal{C}_{\nu}(t):=\sum_{k:\;x_k\leq t}e_k\sqrt{y_k},\quad t\geq 0,
\end{equation*}
where $\mathcal{P}=\sum_{k}\delta_{(x_k,y_k)}$ is the Poisson process on $[0,\infty)\times (0,\infty)$ with intensity measure $\mathbb{LEB}\times \nu$.
\end{definition}

Note that 
for every fixed $t\geq 0$,
\begin{equation*}
   \|\mathcal{C}_{\nu}(t)\|^2_2=\sum_{k:\;x_k\leq t}y_k=\mathcal{S}_{\nu}(t)\in [0,\infty)
\end{equation*}
and therefore $\mathcal{C}_{\nu}(t)$ is a random element of $\ell^2$ a.s., for every $t\geq 0$. Note also that, as a curve, $t\mapsto \mathcal{C}_{\nu}(t)$ is not $\ell^2$-continuous but is a.s.~c\`{a}dl\`{a}g.

In a similar way as the closed range of a subordinator is defined, see~\cite[Section~1.4]{Bertoin}, we define the closed range of a crinkled subordinator.

\begin{definition}\label{def:range}
Fix $T>0$. The range $\mathcal{R}_{\nu}(T)$ of a crinkled subordinator on $[0,T]$ is a random closed subset of $\ell^2$ defined as the closure in $\ell^2$ of the image of $t\mapsto\mathcal{C}_{\nu}(t)$, $t\in [0,T]$. Thus,
\begin{equation*}
   \mathcal{R}_{\nu}(T)={\rm cl}\left(\{\mathcal{C}_{\nu}(t):	0\leq t\leq T\}\right)=\{\mathcal{C}_{\nu}(t):	0\leq t\leq T\}\cup \{\mathcal{C}_{\nu}(t-):	0\leq t\leq T\}.
\end{equation*}
\end{definition}
According to Lemma~\ref{eq:compactness} below, the set $ \mathcal{R}_{\nu}(T)$ is a.s.~compact in $\ell^2$ for every fixed $T>0$. In particular, $(\mathcal{R}_{\nu}(T),\|\cdot\|_2)$ is a compact metric space. Moreover, this space (up to isometries of $\ell^2$)\ does not depend on the choice of an orthonormal basis of $\ell^2$, whereas the  crinkled subordinator itself does depend on the basis. As a metric space, $(\mathcal{R}_{\nu}(T),\rho_2)$ is isometric to the closed range of the  subordinator $(\mathcal{S}_{\nu}(t))_{t\in [0,T]}$ given by
\begin{equation}\label{eq:one_dimensional_range}
\widetilde{\mathcal{R}}_{\nu}(T):={\rm cl}\left(\{\mathcal{S}_{\nu}(t):	0\leq t\leq T\}\right)=\{\mathcal{S}_{\nu}(t):	0\leq t\leq T\}\cup \{\mathcal{S}_{\nu}(t-):	0\leq t\leq T\},
\end{equation}
and endowed with the metric $(t,s)\mapsto \sqrt{|t-s|}$. An isometry $\varphi:\mathcal{R}_{\nu}(T)\to \widetilde{\mathcal{R}}_{\nu}(T)$ is given by
\begin{equation}\label{eq:isometry_with_one_dim}
\varphi(\mathcal{C}_{\nu}(t))=\|\mathcal{C}_{\nu}(t)\|_2^2=\mathcal{S}_{\nu}(t),\quad \varphi(\mathcal{C}_{\nu}(t-))=\|\mathcal{C}_{\nu}(t-)\|_2^2=\mathcal{S}_{\nu}(t-),\quad t\in [0,T].
\end{equation}
Among other things, this implies that the Hausdorff dimension of $\mathcal{R}_{\nu}(T)$ is equal to twice the Hausdorff dimension of $\widetilde{\mathcal{R}}_{\nu}(T)$; see Proposition~\ref{prop:Haus_dim} below. A formula for the Hausdorff dimension of $\widetilde{\mathcal{R}}_{\nu}(T)$ is available; see Section 5.1.2 in~\cite{Bertoin}.

\subsection{Main results}
For the sequence of random walks given by~\eqref{eq:rw_def} and satisfying assumptions~\ref{it(a)}, \ref{it(b)} and \ref{it(c)} above, define a sequence of finite random metric subspaces of $\ell^2$ via
\begin{equation*}
   \mathcal{M}^{(d)}_k:=\left\{\frac{S_0^{(d)}}{\sqrt{a(n)}},\frac{S_1^{(d)}}{\sqrt{a(n)}},\ldots,\frac{S_k^{(d)}}{\sqrt{a(n)}}\right\},\quad d\in\N,\quad k\in\N.
\end{equation*}
Each $\mathcal{M}^{(d)}_k$ is endowed with the induced $\ell^2$-metric. Equivalently, since $\mathcal{M}^{(d)}_k$ lives in
$\R^d$, which we assume to be naturally embedded into $\ell^2$, this induced metric coincides with the standard Euclidean metric.

Here is our main result.

\begin{theorem}\label{thm:main}
Assume that conditions~\ref{it(a)}, \ref{it(b)} and \ref{it(c)} are fulfilled. Fix $T>0$. Then, weakly on the Gromov--Hausdorff space of compact metric spaces, it holds
\begin{equation}\label{eq:main_result}
\left(\mathcal{M}^{(d)}_{\lfloor nT\rfloor},\rho_2\right)~\Longrightarrow (\mathcal{R}_{\nu}(T),\rho_2),\quad d\to\infty.
\end{equation}
\end{theorem}
\begin{rem}
A reminder on the Gromov--Hausdorff space will be given in Section~\ref{sec:GH_convergence}.
\end{rem}

\begin{cor}
Assume that the distribution of $\|X^{(d)}\|$ does not depend on $d$ and satisfies~\eqref{eq:regular_variation1}. Suppose further that $\Theta^{(d)}$ and $\|X^{(d)}\|$ are independent, $\E \Theta^{(d)}=0$ and~\eqref{eq:assump_orthogonal_indep} holds. Then~\eqref{eq:main_result} holds.
\end{cor}

For a compact metric space $M$, denote by $\diam(M)$ its diameter. Since the mapping $M\mapsto \diam(M)$ is continuous with respect to the Gromov--Hausdorff metric, see Exercise~7.3.14 in~\cite{Burago+Burago}, we immediately obtain the following corollary of Theorem~\ref{thm:main}.
\begin{cor}
Under the same assumptions as in Theorem~\ref{thm:main}, we have
\begin{equation*}
   (\diam(\mathcal{M}^{(d)}_{\lfloor nT\rfloor}))^2=\frac{\max_{0\leq i,k\leq \lfloor nT\rfloor}\|S_i^{(d)}-S_k^{(d)}\|_2^2}{a(n)}\Longrightarrow (\diam(\mathcal{R}_{\nu}(T))^2=\mathcal{S}_{\nu}(T),\quad d\to\infty.
\end{equation*}
\end{cor}

\subsection{Examples}

Below we consider three families of random walks satisfying the assumptions~\ref{it(a)}, \ref{it(b)} and \ref{it(c)}.

\begin{example}[Rotationally  invariant distributions] Let $X^{(d)}$ be a random vector in $\R^d$ with a rotationally invariant distribution. This means that $\Theta^{(d)}$ is uniformly distributed on the unit sphere in $\R^d$, and $\|X^{(d)}\|_2$ and $\Theta^{(d)}$ are independent. Assume that the distribution of $\|X^{(d)}\|_2^2$ satisfies~\eqref{eq:regular_variation_gen} and~\eqref{eq:regular_variation_trunc_sec_moment}. Condition~\ref{it(b)} follows from Remark~3.2.5 in~\cite{vershynin_book}, whereas condition~\ref{it(c)} is a consequence of $\E \Theta^{(d)}=0$ and independence of $\|X^{(d)}\|_2$ and $\Theta^{(d)}$.
\end{example}

\begin{example}[Random walks jumping along the coordinate axes]  The following model is similar to the simple random walk on $\Z^d$. Let $\widehat{V}^{(d)}$ be a random vector distributed uniformly on the set $\{\pm e_1,\ldots, \pm e_d\}$, that is  $\mathbb{P}\{\widehat{V}^{(d)} = e_j\}=\mathbb{P}\{\widehat{V}^{(d)} = -e_j\}=1/(2d)$ for all $j\in \{1,\dots, d\}$. For every $d\in\mathbb{N}$, put $X^{(d)}:= R^{(d)} \cdot \widehat{V}^{(d)}$, where $R^{(d)}$ is a positive random variable which is independent of $\widehat{V}^{(d)}$. Assume that the distribution of $\|X^{(d)}\|_2^2=(R^{(d)})^2$ satisfies~\eqref{eq:regular_variation_gen} and~\eqref{eq:regular_variation_trunc_sec_moment}. Condition~\ref{it(b)} holds automatically, since $\Theta^{(d)}=\widehat{V}^{(d)}$,
and therefore
\begin{equation*}
   \mmp\left\{\left\langle \Theta^{(d)}_1,\Theta^{(d)}_2\right\rangle\neq 0\right\}=\mmp\left\{\left\langle \widehat{V}_1^{(d)},\widehat{V}_2^{(d)}\right\rangle\neq 0\right\}=\frac{1}{d}\to 0,\quad d\to\infty.
\end{equation*}
Condition~\ref{it(c)} also holds automatically since $\E \widehat{V}^{(d)}=0$.
\end{example}

\begin{example}[Random walks with i.i.d.\ symmetric heavy-tailed components]\label{ex3}
Let $\xi$ be a symmetric random variable such that for some $\alpha \in (0,1)$,
\begin{equation}\label{eq:regular_variation_ex3}
\mathbb{P}\{\xi^2>t\}~\sim~ t^{-\alpha},\quad t\to\infty.
\end{equation}
For an array $(\xi_{i,j})_{i,j\in\mathbb{N}}$ of independent copies of $\xi$, put
\begin{equation*}
   X^{(d)}_i:=d^{-1/(2\alpha)}(\xi_{i,1},\xi_{i,2},\ldots,\xi_{i,d}),\quad i\in\N,\quad d\in\N,
\end{equation*}
and let $X^{(d)}$ be a generic copy of $X^{(d)}_i$. According to Eq.~(4) in~\cite{Heyde:1968}, for every fixed $x>0$,
\begin{equation}\label{eq:heyde_ex3}
\lim_{d\to\infty}n\mmp\{\|X^{(d)}\|_2^2>xn^{1/\alpha}\}=\lim_{d\to\infty}n\cdot d\cdot\mmp\{\xi^2>xd^{1/\alpha}n^{1/\alpha}\}=x^{-\alpha}.
\end{equation}
Thus,~\eqref{eq:regular_variation_gen} holds with $\nu(x,\infty)=x^{-\alpha}$ and $a(n)=n^{1/\alpha}$. 
Condition~\eqref{eq:regular_variation_trunc_sec_moment} follows from the following chain of estimates:
\begin{multline*}
\frac{n}{a(n)}\E\left(\|X^{(d)}_1\|_2^2\ind_{\{\|X^{(d)}_1\|_2^2\leq s a(n)\}}\right)=\frac{n}{(nd)^{1/\alpha}}\E\left(\left(\sum_{j=1}^{d}\xi_{1,j}^2\right)\ind_{\{\sum_{j=1}^{d}\xi_{1,j}^2\leq s (nd)^{1/\alpha}\}}\right)\\
\leq (nd)^{1-1/\alpha}\E\left(\xi_{1,1}^2\ind_{\{\xi_{1,1}^2\leq s(nd)^{1/\alpha}\}}\right)~\sim~\frac{\alpha}{1-\alpha}s^{1-\alpha},\quad d\to\infty,
\end{multline*}
where the last asymptotic equivalence is a consequence of~\eqref{eq:regular_variation_ex3}; see p.~579 in~\cite{Feller_Vol2}. The right-hand side of the last display converges to zero as $s\to 0+$, which yields~\eqref{eq:regular_variation_trunc_sec_moment}. Condition~\eqref{eq:centering_is_negligible} follows from the symmetry of $X^{(d)}$, which is inherited from the symmetry of $\xi$. For a proof (standard but technical)\ of part~\ref{it(b)}, we refer the reader to Lemma~\ref{lem:example3_verification} in the Appendix.

\end{example}

\subsection{Compact subsets of \texorpdfstring{$\ell^2$}{l2} and their convergence}\label{sec:GH_convergence}

In this subsection, we recall the definitions of the Hausdorff and Gromov--Hausdorff metrics, and the notion of Hausdorff distance up to isometry in $\ell^2$.

\subsubsection{Hausdorff and Gromov--Hausdorff metrics}
Let $(M,\rho)$ be an arbitrary metric space. Denote by $\mathcal{K}(M)$ the set of non-empty compact subsets of $M$. Let $d_{H}$ denote the Hausdorff distance between elements of $\mathcal{K}(M)$, defined by
\begin{equation*}
   d_{H}(A,B) = \inf\{r>0: A\subset U_r(B), B\subset U_r(A)\}.
\end{equation*}
Here, $U_r(A) = \{m\in M: \rho(A, m) < r\}$ is the $r$-neighborhood of $A$ in $M$. It is well known that $(\mathcal{K}(M),d_{H})$ is a metric space. If $M$ is complete, then $(\mathcal{K}(M),d_{H})$ is also complete.

By an isometry between two sets $E_1$ and $E_2$ living in possibly different metric spaces $(M_1,\rho_1)$ and $(M_2,\rho_2)$, we understand a bijective mapping $J:E_1\to E_2$ such that $\rho_2(J(x),J(y))=\rho_1(x,y)$, for all $x,y\in M_1$. The \emph{Gromov--Hausdorff distance} $d_{GH}(E_1,E_2)$ between two compact metric spaces $E_1$ and $E_2$ is defined as the infimum of $d_{H}(\varphi_1(E_1),\varphi_2(E_2))$, where the infimum is taken over all metric spaces $(M,\rho)$ and all isometric embeddings (injective isometries)\ $\varphi_1: E_1\to M$ and $\varphi_2:E_2\to M$. It is known that the set of isometry classes of compact metric spaces, endowed with the Gromov--Hausdorff distance, becomes a complete separable metric space, called the \emph{Gromov--Hausdorff space}.

\subsubsection{Hausdorff distance up to isometry in \texorpdfstring{$\ell^2$}{l2}}

Compact metric spaces we are interested in live in the same Hilbert space $\ell^2$. This suggests that the full power of the general notion of Gromov--Hausdorff distance might not be needed. Addressing this question, the following notion of closeness between compact subsets of $\ell^2$, called {\it Hausdorff distance up to isometry}, has been proposed in~\cite{KabMar:2022}. Note that a very close concept can be found in Exercise 5.26 in~\cite{Petrunin_lectures}.

Introduce the following equivalence relation $\sim$ on $\mathcal{K}(\ell^2)$, the collection of non-empty compact subsets of $\ell^2$. Two subsets $K_1\subset \ell^2$ and $K_2\subset \ell^2$ are considered equivalent if there is an isometry $J:K_1\to K_2$, that is a bijection between $K_1$ and $K_2$ that preserves distances. The equivalence class of a compact set $K$ is denoted by $[K]:=\{K'\in \mathcal{K}(\ell^2): K\sim K'\}$. The set of all such equivalence classes is denoted by $\mathbb{H} := \mathcal{K}(\ell^2)/\sim$. Now we introduce a metric on $\mathbb{H}$. For $K_1,K_2\in \mathcal{K}(\ell^2)$, the \textit{Hausdorff distance up to isometry} between $[K_1]$ and $[K_2]$ is defined by
\begin{equation}\label{eq:def_up_to_isometry}
d_{\sim}([K_1], [K_2]) = \inf_{K_1'\in [K_1], K_2'\in [K_2]} d_{H}(K_1', K_2').
\end{equation}
\begin{prop}[Proposition 2.10 in~\cite{KabMar:2022}]\label{prop:tilde_d_is_a_metric}
The function $d_{\sim}:\mathbb{H}\times \mathbb{H}\mapsto [0,\infty)$ is a metric on $\mathbb{H}$.
\end{prop}

The following result establishes equivalence of Gromov--Hausdorff convergence and convergence in $(\mathbb{H},d_{\sim})$.

\begin{theorem}[Theorem 2.12 in~\cite{KabMar:2022}]\label{theo:equivalence_GH_and_up_to_iso}
Let $K_1,K_2,\ldots$ and $K$ be compact subsets of $\ell^2$. Then, $[K_n] \to [K]$ in $(\mathbb{H}, d_{\sim})$ if and only if $K_n\to K$ in the Gromov--Hausdorff sense (where $K_n$ and $K$ are regarded as metric spaces with the induced $\ell^2$-metric).
\end{theorem}

\begin{rem}
The notion used in~\cite{Petrunin_lectures}  differs from our definition~\eqref{eq:def_up_to_isometry} by two aspects. Firstly, the space $\ell^2$ is replaced by a universal homogeneous metric space (Urysohn space)\ $\mathcal{U}_{\infty}$. Secondly, the infimum used in the definition of $d_{\sim}$ is taken over {\it global} isometries of $\mathcal{U}_{\infty}$. This defines a pseudometric on the family of compact subsets on $\mathcal{\mathcal{U}_{\infty}}$, for which the corresponding metric space is isometric to the Gromov--Hausdorff space.
\end{rem}

In view of Theorem~\ref{theo:equivalence_GH_and_up_to_iso}, we immediately obtain the following:
\begin{cor}
Under the assumptions of Theorem~\ref{thm:main}, the following holds true:
\begin{equation}
\label{eq:main_result_d_iso}
   [\mathcal{M}^{(d)}_{\lfloor nT\rfloor}]~\Longrightarrow~[\mathcal{R}_{\nu}(T)],\quad d\to\infty,
\end{equation}
weakly on the space of probability measures on $(\mathbb{H},d_{\sim})$.
\end{cor}

Let $\conv$ (respectively, $\oconv$)\ denote the operation of taking convex (respectively, closed convex)\ hull. Lemma~4.3 in~\cite{KabMar:2022} and the continuous mapping theorem yield the following result.

\begin{theorem}\label{theo:conv_to_convex_hull_of_wiener_spiral_in_hilbert_up_to_isometry}
Under the assumptions of Theorem~\ref{thm:main}, the following holds true:
\begin{equation*}
[\conv\mathcal{M}^{(d)}_{\lfloor nT\rfloor}]~\Longrightarrow~[\oconv\mathcal{R}_{\nu}(T)],\quad d\to\infty,
\end{equation*}
weakly on the space of probability measures on $(\mathbb{H},d_{\sim})$.
\end{theorem}

The limiting closed convex hull can be characterized as follows. Let $G_{\downarrow}$ denote the set of nonincreasing functions $g:[0,T]\to [0,1]$. Then
\begin{equation*}
   \oconv\mathcal{R}_{\nu}(T)=\left\{\sum_{k: x_k \leq T}e_k\sqrt{y_k}g(x_k):g\in G_{\downarrow}\right\}.
\end{equation*}

\section{Proof of Theorem~\ref{thm:main}}
Fix $s>0$, define the truncated variables 
\begin{equation*}
   X_k^{(d)}(s):=X_k^{(d)}\ind_{\{\|X_k^{(d)}\|_2^2\geq s a(n)\}},\quad k\in\N,
\end{equation*}
the corresponding random walk
\begin{equation*}
   S_0^{(d)}(s):=0,\quad S_k^{(d)}(s):=X_1^{(d)}(s)+X_2^{(d)}(s)+\cdots+X_k^{(d)}(s),\quad k\in\N,
\end{equation*}
and the sets
\begin{equation*}
   \mathcal{M}_k^{(d)}(s):=\left\{\frac{S_0^{(d)}(s)}{\sqrt{a(n)}},\frac{S_1^{(d)}(s)}{\sqrt{a(n)}},\ldots,\frac{S_k^{(d)}(s)}{\sqrt{a(n)}}\right\},\quad k\in\N,
\end{equation*}
which we regard as a.s.~finite metric spaces endowed with the induced $\rho_2$ metric.

Define the random set
\begin{equation}\label{eq:eps_net_definition}
\mathcal{R}_{\nu}^{(s)}(T):=\left\{\sum_{k:\;x_k\leq t}e_k\sqrt{y_k} \ind_{\{y_k>s\}}: 0\leq t\leq T\right\},
\end{equation}
c.f.\ Definition~\ref{def:range}, and note that it is a.s.~finite for every fixed $s>0$, since there are a.s.~finitely many points $(x_k,y_k)$ of $\mathcal{P}$ in $[0,T]\times (s,\infty)$.

Since the Gromov--Hausdorff space is complete and separable, according to Theorem 3.2 on p.~28 in \cite{Billingsley:1999} it suffices to check that  the following three relations hold true:
\begin{equation}\label{eq:bill1}
(\mathcal{M}_{\lfloor nT\rfloor}^{(d)}(s),\rho_2)\Longrightarrow (\mathcal{R}_{\nu}^{(s)}(T),\rho_2),\quad d\to\infty,
\end{equation}
for every fixed $s>0$, weakly on the Gromov--Hausdorff space;
\begin{equation}\label{eq:bill2}
(\mathcal{R}_{\nu}^{(s)}(T),\rho_2)\longrightarrow (\mathcal{R}_{\nu}(T),\rho_2),\quad s\to 0+,
\end{equation}
a.s.\ on the Gromov--Hausdorff space; and
\begin{equation}\label{eq:bill3}
\lim_{s\to 0+}\limsup_{d\to\infty}\mmp\{d_{GH}(\mathcal{M}_{\lfloor nT\rfloor}^{(d)}(s),\mathcal{M}_{\lfloor nT\rfloor}^{(d)})>\varepsilon\}=0,
\end{equation}
for every fixed $\varepsilon>0$. The easiest relation to prove among~\eqref{eq:bill1}, \eqref{eq:bill2} and~\eqref{eq:bill3} is the second one: it follows from Lemma~\ref{lem:finite_approx_to_R(T)} below by using the obvious inequality $d_{GH}\leq d_H$.

\vspace{5mm}

\noindent
{\sc Proof of~\eqref{eq:bill1}.} The convergence stated in~\eqref{eq:bill1} is the weak convergence of probability measures on the  Gromov--Hausdorff space. A natural way to deal with it could be working with the Gromov--Hausdorff--Prohorov metric; see~\cite{Abraham+Delmas+Hoscheit:2013,Andreas+Pfaffelhuber+Winter:2009,Miermont:2009}. However, in our setting we are able to avoid this heavy machinery by an appeal to a version of the Skorokhod representation theorem.

For $k\in\N$, put $R^{(d)}_k:=\|X^{(d)}_k\|_2$. Let $M_p:=M_p([0,\infty)\times(0,\infty])$ be the space of locally finite point measures on $[0,\infty)\times(0,\infty]$, endowed with the vague topology. This space is known to be complete and separable; see Proposition~3.17 in~\cite{Resnick:2008}. Furthermore, under the assumption~\eqref{eq:regular_variation_gen}, the following convergence in distribution on $M_p$ holds true:
\begin{equation*}
   \sum_{k\geq 1}\delta_{(k/n,(R^{(d)}_k)^2/a(n))}\Longrightarrow \mathcal{P},\quad d\to\infty;
\end{equation*}
see Proposition~3.21 in the same reference.  Now we want to apply a version of Skorokhod's representation theorem, which will allow us to pass to a new probability space on which the distributional convergence above can be replaced by the a.s.\ convergence. Note that the left-hand side of the latter formula can be viewed as an image of a measurable map $\phi_n$ from $(\ell^2)^{\N}$ to $M_p$, defined by
\begin{equation*}
   \phi_n(X^{(d)}_1,X^{(d)}_2,\ldots)=\sum_{k\geq 1}\delta_{(k/n,\|X^{(d)}_k\|_2^2/a(n))}.
\end{equation*}
Thus, applying an extended version of the Skorokhod representation theorem, stated in Lemma~\ref{lem:skorokhod}, we can pass to a new probability space $(\overline{\Omega},\overline{\mathcal{F}},\overline{\mmp})$ which accommodates the following  objects:
\begin{itemize}
   \item for every $d\in\N$, a distributional copy $(\overline{X}_k^{(d)})_{k\in\N}$ of the sequence $(X_k^{(d)})_{k\in\N}$;
   \item a distributional copy $\overline{\mathcal{P}}:=\sum_k \delta_{(\overline{x}_k,\overline{y}_k)}$ of the Poisson point process $\mathcal{P}$;
\end{itemize}
such that with $\overline{R}^{(d)}_k:=\|\overline{X}^{(d)}_k\|_2$, $k\in\N$, it holds
\begin{equation}\label{eq:bill1_proof_as}
\overline{\mathcal{P}}_n:=\sum_{k\geq 1}\delta_{(k/n,(\overline{R}^{(d)}_k)^2/a(n))}\longrightarrow \sum_k \delta_{(\overline{x}_k,\overline{y}_k)},\quad \overline{\mmp}-{\rm a.s.}\quad \text{as}\quad d\to\infty.
\end{equation}

Define $\overline{\mathcal{M}}^{(d)}_{k}(s)$ and $\overline{\mathcal{R}}^{(s)}_{\nu}(T)$ in the obvious manner via $(\overline{X}_k^{(d)})_{k\in\N}$ and $\overline{\mathcal{P}}$, respectively. More precisely, put
\begin{equation*}
   \overline{\mathcal{M}}^{(d)}_{k}(s):=\left\{\sum_{j=1}^{\ell}\frac{\overline{X}_j^{(d)}\ind_{\{(\overline{R}^{(d)}_j)^2\geq s a(n)\}}}{\sqrt{a(n)}}:\ell=0,\ldots,k\right\},\quad k\in\N,
\end{equation*}
and
\begin{equation*}
   \overline{\mathcal{R}}^{(s)}_{\nu}(T):=\left\{\sum_{k:\;\overline{x}_k\leq t}e_k\sqrt{\overline{y}_k} \ind_{\{\overline{y}_k>s\}}: 0\leq t\leq T\right\}.
\end{equation*}
We shall prove that~\eqref{eq:bill1_proof_as} yields
\begin{equation}\label{eq:bill1_prob}
d_{GH}(\overline{\mathcal{M}}_{\lfloor nT\rfloor}^{(d)}(s),\overline{\mathcal{R}}_{\nu}^{(s)}(T))\overset{\overline{\mmp}}{\longrightarrow} 0,\quad d\to\infty,
\end{equation}
in the Gromov--Hausdorff space, for every fixed $s>0$. The latter is clearly sufficient for~\eqref{eq:bill1}, since
\begin{equation*}
   \mathbb{E}f(\mathcal{M}_{\lfloor nT\rfloor}^{(d)}(s))=\overline{\mathbb{E}}f(\overline{\mathcal{M}}_{\lfloor nT\rfloor}^{(d)}(s))~\longrightarrow~\overline{\mathbb{E}}f(\overline{\mathcal{R}}_{\nu}^{(s)}(T))=\mathbb{E}f(\mathcal{R}_{\nu}^{(s)}(T)),\quad d\to\infty,
\end{equation*}
for every bounded continuous $f$, with $\overline{\mathbb{E}}$ denoting the expectation with respect to $\overline{\mathbb{P}}$.

Let $\overline{\Omega}^{\prime}$ be an event of probability one on the new probability space such that~\eqref{eq:bill1_proof_as} holds for all $\overline{\omega}\in\overline{\Omega}^{\prime}$, and fix any $\overline{\omega}\in\overline{\Omega}^{\prime}$. For notational simplicity, we suppress the dependence on $\overline{\omega}$ below. According to Proposition~3.13 in~\cite{Resnick:2008}, there exist an integer $P=P(\overline{\omega})\in\N$ and an enumeration of the atoms of $\overline{\mathcal{P}}$ and $\overline{\mathcal{P}}_n$ in $[0,T]\times [s,\infty)$ such that for all sufficiently large $n\in\N$,
\begin{equation*}
   \overline{\mathcal{P}}_n(\cdot \cap ([0,T]\times [s,\infty)))=\sum_{j=1}^{P}\delta_{(k_j(n)/n,(\overline{R}^{(d)}_{k_j(n)})^2/a(n))}\quad\text{and}\quad\overline{\mathcal{P}}(\cdot \cap ([0,T]\times [s,\infty)))=\sum_{j=1}^{P}\delta_{(\overline{x}_{k_j},\overline{y}_{k_j})},
\end{equation*}
and, moreover,
\begin{equation}\label{eq:bill1_proof_pointwise}
\lim_{d\to\infty}\left(\frac{k_j(n)}{n},\frac{(\overline{R}^{(d)}_{k_j(n)})^2}{a(n)}\right)= (\overline{x}_{k_j},\overline{y}_{k_j}),\quad j=1,\ldots,P.
\end{equation}
Without loss of generality, we assume that the enumeration is chosen such that $\overline{x}_{k_1}<\overline{x}_{k_2}<\cdots<\overline{x}_{k_P}$. Then it is  clear  that
\begin{equation*}
   \overline{\mathcal{M}}^{(d)}_{\lfloor nT\rfloor}(s)=\left\{\sum_{j=1}^{\ell}\frac{\overline{X}_{k_j(n)}^{(d)}}{\sqrt{a(n)}}:\ell=0,\ldots,P\right\}.
\end{equation*}
Also,
\begin{equation*}
   \overline{\mathcal{R}}^{(s)}_{\nu}(T)=\left\{\sum_{j=1}^{\ell}e_{k_j}\sqrt{\overline{y}_{k_j}}:\ell=0,\ldots,P\right\}.
\end{equation*}
Define the bijective mapping
\begin{equation*}
   I_n:\overline{\mathcal{R}}^{(s)}_{\nu}(T)\longmapsto \overline{\mathcal{M}}^{(d)}_{\lfloor nT\rfloor}(s)
\end{equation*}
by
\begin{equation*}
   I_n\left(\sum_{j=1}^{\ell}e_{k_j}\sqrt{\overline{y}_{k_j}}\right)=\sum_{j=1}^{\ell}\frac{\overline{X}_{k_j(n)}^{(d)}}{\sqrt{a(n)}},\quad \ell=0,\ldots,P.
\end{equation*}
By Corollary~7.3.28 on p.~258 of \cite{Burago+Burago}, the Gromov--Hausdorff distance between $\overline{\mathcal{M}}^{(d)}_{\lfloor nT\rfloor}(s)$ and $\overline{\mathcal{R}}^{(s)}_{\nu}(T)$ is
bounded above by twice the distortion of the map $I_n$, that is
\begin{align*}
&\hspace{-1cm}d_{GH}(\overline{\mathcal{M}}^{(d)}_{\lfloor nT\rfloor}(s),\overline{\mathcal{R}}^{(s)}_{\nu}(T))\\
&\leq
2\sup_{0\leq \ell\leq m\leq P}\left|\Big\|I_n\left(\sum_{j=1}^{m}e_{k_j}\sqrt{\overline{y}_{k_j}}\right)-I_n\left(\sum_{j=1}^{\ell}e_{k_j}\sqrt{\overline{y}_{k_j}}\right)\Big\|_2-\Big\|\sum_{j=1}^{m}e_{k_j}\sqrt{\overline{y}_{k_j}}-\sum_{j=1}^{\ell}e_{k_j}\sqrt{\overline{y}_{k_j}}\Big\|_2\right|\\
&=2\sup_{0\leq \ell\leq m\leq P}\left|\Big\|\sum_{j=\ell+1}^{m}\frac{\overline{X}_{k_j(n)}^{(d)}}{\sqrt{a(n)}}\Big\|_2-\Big\|\sum_{j=\ell+1}^{m}e_{k_j}\sqrt{\overline{y}_{k_j}}\Big\|_2\right|.
\end{align*}
Note that
\begin{align*}
&\sup_{0\leq \ell\leq m\leq P}\left|\Big\|\sum_{j=\ell+1}^{m}\frac{\overline{X}_{k_j(n)}^{(d)}}{\sqrt{a(n)}}\Big\|_2-\Big\|\sum_{j=\ell+1}^{m}e_{k_j}\sqrt{\overline{y}_{k_j}}\Big\|_2\right|^2
\\
&~~~~~\leq
\sup_{0\leq \ell\leq m\leq P}\left|\Big\|\sum_{j=\ell+1}^{m}\frac{\overline{X}_{k_j(n)}^{(d)}}{\sqrt{a(n)}}\Big\|_2^2-\Big\|\sum_{j=\ell+1}^{m}e_{k_j}\sqrt{\overline{y}_{k_j}}\Big\|_2^2\right|
\\
&~~~~~=
\sup_{0\leq \ell\leq m\leq P}\left| \sum_{j=\ell+1}^{m} \left(\frac{(\overline{R}^{(d)}_{k_j(n)})^2}{a(n)} - \overline{y}_{k_j}\right) +\frac{1}{a(n)}\sum_{i,j\in\{\ell+1,\ldots,m\},i\neq j}\langle \overline{X}_{k_i(n)}^{(d)},\overline{X}_{k_j(n)}^{(d)}\rangle_2\right|.
\end{align*}
In the first step we used the inequality $|x-y|^2 \leq |x-y||x+y| = |x^2-y^2|$ for $x,y\geq 0$. In view of~\eqref{eq:bill1_proof_pointwise}, the first sum on  the right-hand side converges to $0$ for all $\overline{\omega}\in \overline{\Omega}^{\prime}$. Therefore, it remains  to check that
\begin{equation}\label{eq:bill1_proof4}
\frac{1}{a(n)}\sup_{0\leq \ell\leq m\leq P}\left| \sum_{i,j\in\{\ell+1,\ldots,m\},i\neq j}\langle \overline{X}_{k_i(n)}^{(d)},\overline{X}_{k_j(n)}^{(d)}\rangle_2 \right|\overset{\overline{\mmp}}{\longrightarrow}0,\quad d\to\infty.
\end{equation}
Using that
\begin{align*}
\left|\sum_{i,j\in\{\ell+1,\ldots,m\},i\neq j}\langle \overline{X}_{k_i(n)}^{(d)},\overline{X}_{k_j(n)}^{(d)}\rangle_2\right|&\leq P^2 \sup_{i,j\in\{\ell+1,\ldots,m\},i\neq j}\left|\langle \overline{X}_{k_i(n)}^{(d)},\overline{X}_{k_j(n)}^{(d)}\rangle_2\right|\\
&\leq P^2 \sup_{i,j\in\{1,\ldots,\lfloor nT\rfloor\},i\neq j}\left|\langle \overline{X}_{i}^{(d)}\ind_{\{\|\overline{X}_i^{(d)}\|_2^2\geq s a(n)\}},\overline{X}_{j}^{(d)}\ind_{\{\|\overline{X}_j^{(d)}\|_2^2\geq s a(n)\}}\rangle_2\right|,
\end{align*}
and recalling that $(\overline{X}_k^{(d)})_{k\in\N}$ is a distributional copy of $(X_k^{(d)})_{k\in\N}$, we see that~\eqref{eq:bill1_proof4} is a consequence of
\begin{equation}\label{eq:bill1_proof5}
\frac{1}{a(n)}\sup_{i,j\in\{1,\ldots,\lfloor nT\rfloor\},i\neq j}\left|\langle X_{i}^{(d)}\ind_{\{\|X_i^{(d)}\|_2^2\geq s a(n)\}},X_{j}^{(d)}\ind_{\{\|X_j^{(d)}\|_2^2\geq s a(n)\}}\rangle_2\right|\overset{\mmp}{\longrightarrow}0,\quad d\to\infty.
\end{equation}
To check the latter we note that, for every fixed $\vartheta>0$,
\begin{align*}
&\mmp\left\{\sup_{i,j\in\{1,\ldots,\lfloor nT\rfloor\},i\neq j}\left|\langle X_{i}^{(d)}\ind_{\{\|X_i^{(d)}\|_2^2\geq s a(n)\}},X_{j}^{(d)}\ind_{\{\|X_j^{(d)}\|_2^2\geq s a(n)\}}\rangle_2\right|>\vartheta a(n)\right\}\\
&~~~~~\leq T^2 n^2 \mmp\{\|X_1^{(d)}\|_2^2\geq s a(n),\|X_2^{(d)}\|_2^2\geq s a(n),|\langle X_{1}^{(d)},X_{2}^{(d)}\rangle_2|>\vartheta a(n)\}\\
&~~~~~= T^2 n^2 \mmp\{\|X_1^{(d)}\|_2^2\geq s a(n),\|X_2^{(d)}\|_2^2\geq s a(n),\|X_1^{(d)}\|_2 \|X_2^{(d)}\|_2\geq Sa(n), |\langle X_{1}^{(d)},X_{2}^{(d)}\rangle_2|>\vartheta a(n)\}\\
&~~~~~+ T^2 n^2 \mmp\{\|X_1^{(d)}\|_2^2\geq s a(n),\|X_2^{(d)}\|_2^2\geq s a(n),\|X_1^{(d)}\|_2 \|X_2^{(d)}\|_2< Sa(n), |\langle X_{1}^{(d)},X_{2}^{(d)}\rangle_2|>\vartheta a(n)\},
\end{align*}
where $S>s$ is fixed. Note that by~\eqref{eq:regular_variation_gen}
\begin{equation}
n^2\mmp\{(a(n))^{-1}(\|X_1^{(d)}\|_2^2,\|X_2^{(d)}\|_2^2)\in \cdot\}\tovague (\nu\otimes\nu)(\cdot),
\end{equation}
on $(0,\infty]\times (0,\infty]$ and, therefore,
\begin{multline*}
\lim_{d\to\infty}n^2 \mmp\{\|X_1^{(d)}\|_2^2\geq s a(n),\|X_2^{(d)}\|_2^2\geq s a(n),\|X_1^{(d)}\|_2 \|X_2^{(d)}\|_2\geq Sa(n)\}\\
=(\nu\otimes\nu)\left([s,\infty)\times [s,\infty)\cap \{(x,y)\in (0,\infty]^2:xy\geq S^2\}\right),
\end{multline*}
for all but countably many $S>s$. The right-hand side converges to zero as $S\to\infty$. On the other hand,
\begin{align*}
&\hspace{-0.5cm}n^2 \mmp\left\{\|X_1^{(d)}\|_2^2\geq s a(n),\|X_2^{(d)}\|_2^2\geq s a(n),\|X_1^{(d)}\|_2 \|X_2^{(d)}\|_2< Sa(n), |\langle X_{1}^{(d)},X_{2}^{(d)}\rangle_2|>\vartheta a(n)\right\}\\
&\leq n^2 \mmp\left\{\|X_1^{(d)}\|_2^2\geq s a(n),\|X_2^{(d)}\|_2^2\geq s a(n),\|X_1^{(d)}\|_2 \|X_2^{(d)}\|_2< Sa(n),\left|\left\langle \Theta_1^{(d)},\Theta_2^{(d)}\right\rangle_2\right|>\vartheta S^{-1}\right\}\\
&\leq n^2 \mmp\left\{\|X_1^{(d)}\|_2^2\geq s a(n),\|X_2^{(d)}\|_2^2\geq s a(n),\left|\left\langle \Theta_1^{(d)},\Theta_2^{(d)}\right\rangle_2\right|>\vartheta S^{-1}\right\}\\
&\leq C(s) \mmp\left\{\left|\left\langle \Theta_1^{(d)},\Theta_2^{(d)}\right\rangle_2\right|>\vartheta S^{-1}\Big|\|X_1^{(d)}\|_2^2\geq s a(n),\|X_2^{(d)}\|_2^2\geq s a(n)\right\},
\end{align*}
where $C(s)$ is a positive constant which depends on $s$. The last  passage is again a consequence of~\eqref{eq:regular_variation_gen}. As $d\to\infty$, the conditional probability on the right-hand side of the last  display converges to zero by~\eqref{eq:assump_orthogonal}. This completes the proof of~\eqref{eq:bill1_proof5} as well as of~\eqref{eq:bill1}.

\vspace{5mm}

\noindent
{\sc Proof of~\eqref{eq:bill3}.} It is clear that
\begin{equation*}
   d_{GH}(\mathcal{M}_{\lfloor nT\rfloor}^{(d)}(s),\mathcal{M}_{\lfloor nT\rfloor}^{(d)})\leq d_{H}(\mathcal{M}_{\lfloor nT\rfloor}^{(d)}(s),\mathcal{M}_{\lfloor nT\rfloor}^{(d)})\leq \frac{1}{\sqrt{a(n)}}\max_{k=1,\ldots,\lfloor nT\rfloor}\left\|\sum_{j=1}^{k}X_j^{(d)}\ind_{\{\|X_j^{(d)}\|_2^2\leq sa(n)\}}\right\|_2.
\end{equation*}
Further,
\begin{multline*}
\max_{k=1,\ldots,\lfloor nT\rfloor}\left\|\sum_{j=1}^{k}X_j^{(d)}\ind_{\{\|X_j^{(d)}\|_2^2\leq sa(n)\}}\right\|_2
\leq \max_{k=1,\ldots,\lfloor nT\rfloor}\left\|\sum_{j=1}^{k}\left(X_j^{(d)}\ind_{\{\|X_j^{(d)}\|_2^2\leq sa(n)\}}-\E X_j^{(d)}\ind_{\{\|X_j^{(d)}\|_2^2\leq sa(n)\}}\right)\right\|_2\\
+\max_{k=1,\ldots,\lfloor nT\rfloor}\left\|\sum_{j=1}^{k}\E \left(X_j^{(d)}\ind_{\{\|X_j^{(d)}\|_2^2\leq sa(n)\}}\right)\right\|_2=:Z_{1}^{(d)}(s)+Z_{2}^{(d)}(s).
\end{multline*}
In order to estimate $Z_{1}^{(d)}(s)$, we note that
\begin{equation*}
   \sum_{j=1}^{k}\left(X_j^{(d)}\ind_{\{\|X_j^{(d)}\|_2^2\leq sa(n)\}}-\E X_j^{(d)}\ind_{\{\|X_j^{(d)}\|_2^2\leq sa(n)\}}\right),\quad k\in\N,
\end{equation*}
is an $\R^d$-valued martingale with respect to the natural filtration of the sequence $(X_k^{(d)})_{k\in\N}$. Thus, by Jensen's inequality,
\begin{equation*}
   \left\|\sum_{j=1}^{k}\left(X_j^{(d)}\ind_{\{\|X_j^{(d)}\|_2^2\leq sa(n)\}}-\E X_j^{(d)}\ind_{\{\|X_j^{(d)}\|_2^2\leq sa(n)\}}\right)\right\|_2^2,\quad k\in\N,
\end{equation*}
is a submartingale. By Doob's maximal inequality
\begin{align*}
&\hspace{-1cm}\mmp\left\{Z_{1}^{(d)}(s)\geq 2^{-1}\varepsilon\sqrt{a(n)}\right\}=\mmp\left\{(Z_{1}^{(d)}(s))^2\geq 4^{-1}\varepsilon^2 a(n)\right\}\\
&\leq \frac{4}{\varepsilon^2 a(n)}\E \left\|\sum_{j=1}^{\lfloor nT\rfloor}\left(X_j^{(d)}\ind_{\{\|X_j^{(d)}\|_2^2\leq sa(n)\}}-\E X_j^{(d)}\ind_{\{\|X_j^{(d)}\|_2^2\leq sa(n)\}}\right)\right\|_2^2\\
&= \frac{4\lfloor nT\rfloor}{\varepsilon^2 a(n)}\E\|X^{(d)}\ind_{\{\|X^{(d)}\|_2^2\leq sa(n)\}}-\E X^{(d)}\ind_{\{\|X^{(d)}\|_2^2\leq sa(n)\}}\|_2^2\\
&\leq \frac{8\lfloor nT\rfloor}{\varepsilon^2 a(n)}\E\left(\|X^{(d)}\|_2^2\ind_{\{\|X^{(d)}\|_2^2\leq sa(n)\}}\right).
\end{align*}
Condition~\eqref{eq:regular_variation_trunc_sec_moment} implies that
\begin{equation*}
   \lim_{s\to 0+}\limsup_{d\to\infty}\mmp\left\{Z_{1}^{(d)}(s)\geq 2^{-1}\varepsilon\sqrt{a(n)}\right\}=0.
\end{equation*}
It remains to show that
\begin{equation}\label{eq:bill3_proof1}
\lim_{s\to 0+}\limsup_{d\to\infty}\frac{Z_{2}^{(d)}(s)}{\sqrt{a(n)}}=0.
\end{equation}
Note that
\begin{equation*}
   Z_{2}^{(d)}(s)=\max_{k=1,\ldots,\lfloor nT\rfloor}\left\|\sum_{j=1}^{k}\E \left(X_j^{(d)}\ind_{\{\|X_j^{(d)}\|_2^2\leq sa(n)\}}\right)\right\|_2
\leq Tn \bigl\|\E X^{(d)}\ind_{\{\|X^{(d)}\|_2^2\leq sa(n)\}}\bigr\|_2.
\end{equation*}
Thus,~\eqref{eq:bill3_proof1} follows from~\eqref{eq:centering_is_negligible}.

In Section~\ref{sec:assumptions}, we remarked that if the distribution of $\|X^{(d)}\|^2_2$ does not depend on $d$, then~\eqref{eq:regular_variation1} implies~\eqref{eq:regular_variation_trunc_sec_moment}. Here is the proof of this fact. Using p.~579 in \cite{Feller_Vol2}, Condition~\eqref{eq:regular_variation1} implies
\begin{equation*}
   \E\left(\|X^{(d)}\|_2^2\ind_{\{\|X^{(d)}\|_2^2\leq sa(n)\}}\right)~\sim~\frac{\alpha}{1-\alpha}sa(n)\mmp\{\|X^{(d)}\|_2^2\geq s a(n)\},\quad d\to\infty.
\end{equation*}
Therefore,
\begin{equation}\label{eq:regular_variation_trunc_sec_moment_verify}
\lim_{d\to\infty}\frac{n}{ a(n)}\E\left(\|X^{(d)}\|_2^2\ind_{\{\|X^{(d)}\|_2^2\leq sa(n)\}}\right)=\frac{\alpha s^{1-\alpha}}{(1-\alpha)},
\end{equation}
and the right-hand side converges to zero, as $s\to 0+$.

\section{Properties of the  crinkled subordinator}

\begin{lemma}\label{eq:compactness}
For every $T>0$, the set $\mathcal{R}_{\nu}(T)$ in Definition~\ref{def:range} is a.s.\ compact in $\ell^2$.
\end{lemma}
\begin{proof}
It suffices to show that the set $\{\mathcal{C}_{\nu}(t):	0\leq t\leq T\}$ is a.s.~totally bounded in $\ell^2$, that is, for every $\varepsilon>0$ there exists an a.s.\ finite $\varepsilon$-net for $\{\mathcal{C}_{\nu}(t):	0\leq t\leq T\}$.

By Lebesgue dominated convergence theorem,
\begin{equation}\label{eq:Lebesgue_dominated_convergence}
\lim_{s\to 0+}\sum_{k:\;x_k\leq T}y_k \ind_{\{y_k\leq s\}}=0,\quad \text{a.s.}
\end{equation}
In particular, for every $\varepsilon>0$, there exists a (random) $\delta>0$ such that
\begin{equation*}
   \sum_{k:\;x_k\leq T}y_k \ind_{\{y_k\leq \delta\}}\leq \varepsilon^2.
\end{equation*}
Recall that the set in definition~\eqref{eq:eps_net_definition} is a.s.\ finite, and let us show that $\mathcal{R}_{\nu}^{(\delta)}(T)$ is the sought a.s.~finite $\varepsilon$-net for $\{\mathcal{C}_{\nu}(t):	0\leq t\leq T\}$. For every $t\in [0,T]$, we have
\begin{equation*}
   \left\|C_{\nu}(t)-\sum_{k:\;x_k\leq t}e_k\sqrt{y_k} \ind_{\{y_k>\delta\}}\right\|_2^2=\left\|\sum_{k:\;x_k\leq t}e_k\sqrt{y_k} \ind_{\{y_k\leq \delta\}}\right\|_2^2=\sum_{k:\;x_k\leq t}y_k\ind_{\{y_k\leq \delta\}}
\leq \sum_{k:\;x_k\leq T}y_k\ind_{\{y_k\leq \delta\}}\leq \varepsilon^2,
\end{equation*}
and the proof is complete.
\end{proof}

\begin{lemma}
\label{lem:finite_approx_to_R(T)}
For every $T>0$,
\begin{equation*}
   \lim_{s\to 0+}d_{H}(\mathcal{R}_{\nu}^{(s)}(T)),\mathcal{R}_{\nu}(T))=0,\quad\text{a.s.}
\end{equation*}
\end{lemma}
\begin{proof}
The proof follows from the inequalities
\begin{multline*}
d_{H}(\mathcal{R}_{\nu}^{(s)}(T)),\mathcal{R}_{\nu}(T))\leq \sup_{t\in [0,T]}
\left\|C_{\nu}(t)-\sum_{k:\;x_k\leq t}e_k\sqrt{y_k} \ind_{\{y_k>s\}}\right\|_2\\
=\sup_{t\in [0,T]}\sqrt{\sum_{k:\;x_k\leq t}y_k\ind_{\{y_k\leq s\}}}= \sqrt{\sum_{k:\;x_k\leq T}y_k\ind_{\{y_k\leq s\}}}.
\end{multline*}
In view of~\eqref{eq:Lebesgue_dominated_convergence}, the right-hand side converges to $0$ a.s.\ as $s\to 0+$.
\end{proof}

Let $\mathfrak{d}$ denote the Hausdorff dimension of the set $\widetilde{\mathcal{R}}_{\nu}(T)$ (the closed range of the subordinator $\mathcal S_\nu$; see~\eqref{eq:one_dimensional_range})\ regarded as a subset of $[0,\infty)$ endowed with the Euclidean distance $(t,s)\mapsto|t-s|$. A formula for $\mathfrak{d}$ can be found in Corollary~5.3 of~\cite{Bertoin}.
\begin{prop}\label{prop:Haus_dim}
The Hausdorff dimension of the set $\mathcal{R}_{\nu}(T)$ is a.s.~equal to $2\mathfrak{d}$, for every fixed $T>0$.
\end{prop}
\begin{proof}

By the very definition of the Hausdorff dimension, if a set $A\subset M$ has Hausdorff dimension $\mathfrak{d}$ in a metric space $(M,\rho)$ and $\beta\in (0,1]$ is a fixed parameter, then $A$ has Hausdorff dimension $\mathfrak{d}/\beta$ in the metric space $(M,\rho^{\beta})$. Applying this observation with $\beta=1/2$, we see that as a subset of $[0,\infty)$ endowed with the distance $(t,s)\mapsto\sqrt{|t-s|}$, the set $\widetilde{\mathcal{R}}_{\nu}(T)$ has Hausdorff dimension $2\mathfrak{d}$. It remains to note that isometric sets have the same Hausdorff dimensions and $(\widetilde{\mathcal{R}}_{\nu}(T),\sqrt{|\cdot-\cdot|})$ is isometric to $(\mathcal{R}_{\nu}(T),\rho_2)$; see~\eqref{eq:isometry_with_one_dim}.
\end{proof}

\begin{cor}
If $C_{\nu}$ is a crinkled $\alpha$-stable subordinator with $\alpha \in (0,1)$,  then the Hausdorff dimension of the set $\mathcal{R}_{\nu}(T)$ is a.s.~equal to $2\alpha$, for every fixed $T>0$.
\end{cor}
\begin{proof}
This follows from the fact that the Hausdorff dimension of $\widetilde{\mathcal{R}}_{\nu}(T)$ is equal to $\alpha$; see Theorem 3.2 in~\cite{Blumenthal+Getoor:1960}.
\end{proof}

\begin{rem}
The same arguments in conjunction with the fact that the Wiener spiral is isometric to $[0,1]$ with the metric $(t,s)\mapsto \sqrt{|t-s|}$ demonstrate that the Hausdorff dimension of the Wiener spiral is equal to $2$.
\end{rem}




\section{Appendix}

\subsection{An extension of the Skorokhod representation theorem}

The following lemma is an extended version of the Skorokhod representation theorem. It is a light version of Theorem~1.1 in the paper~\cite{Hu+Bai:2014}; see also~\cite{Verwaat}, where this result appeared for the first time.

\begin{lemma}\label{lem:skorokhod}
Let $(M,\rho)$ and $(M_1,\rho_1)$ be two complete separable metric spaces, and $\phi_n:M\to M_1$ Borel-measurable mappings, $n\in\N$. Suppose that $(\mu_n)_{n\in\N}$ is a sequence of probability measures on $(M,\rho)$ and $\mu_0$ is a probability measure on $(M_1,\rho_1)$ such that $\mu_n\circ\phi_n^{-1}$ converges weakly to $\mu_0$ as $n\to\infty$. Then there exist a sequence of $M$-valued random variables $(X_n)_{n\in\N}$, an $M_1$-valued random variable $X_0$, all defined on a common probability space $(\overline{\Omega},\overline{\mathcal{F}},\overline{P})$, such that $X_n$ has distribution $\mu_n$ for all $n\in\N_0$, and $\phi_n(X_n)\to X_0$ a.s.\ as $n\to\infty$.
\end{lemma}

\subsection{A calculation for Example~\ref{ex3}.}

\begin{lemma}
\label{lem:example3_verification}
In the setting of Example~\ref{ex3}, formula~\eqref{eq:assump_orthogonal} holds true for every fixed $s>0$ and $\varepsilon>0$.
\end{lemma}
\begin{proof}
We start by noting that~\eqref{eq:regular_variation_ex3} implies
\begin{equation*}
   \mmp\{|\xi|>t\}~\sim~t^{-2\alpha},\quad t\to\infty,
\end{equation*}
and, in particular, $\E|\xi|^{\alpha}<\infty$. Recall that $a(n) = n^{1/\alpha}$. 
By conditional Markov's inequality,
\begin{align*}
&\hspace{-1cm}\mmp\left\{\left|\sum_{k=1}^{d}\xi_{1,k}\xi_{2,k}\right|>\varepsilon d^{1/\alpha} \|X_1^{(d)}\|_2\|X_2^{(d)}\|_2\;\Big|\; \|X_1^{(d)}\|_2^2\geq s a(n),\|X_2^{(d)}\|_2^2\geq s a(n)\right\}\\
&\leq \mmp\left\{\left|\sum_{k=1}^{d}\xi_{1,k}\xi_{2,k}\right|> d^{1/\alpha} \varepsilon sa(n)\;\Big|\; \|X_1^{(d)}\|_2^2\geq s a(n),\|X_2^{(d)}\|_2^2\geq s a(n)\right\}\\
&\leq\frac{1}{d(\varepsilon s a(n))^{\alpha}}\E \left(\left|\sum_{k=1}^{d}\xi_{1,k}\xi_{2,k}\right|^{\alpha}\;\Big|\; \|X_1^{(d)}\|_2^2\geq s a(n),\|X_2^{(d)}\|_2^2\geq s a(n)\right)\\
&=\frac{1}{d(\varepsilon s a(n))^{\alpha}}\frac{\E \left(\left|\sum_{k=1}^{d}\xi_{1,k}\xi_{2,k}\right|^{\alpha}\ind_{\{\|X_1^{(d)}\|_2^2\geq s a(n),\|X_2^{(d)}\|_2^2\geq s a(n)\}}\right)}{\mmp\{\|X_1^{(d)}\|_2^2\geq s a(n),\|X_2^{(d)}\|_2^2\geq s a(n)\}}\\
&\sim \frac{s^{-2\alpha}n^2}{d(\varepsilon s a(n))^{\alpha}}\E \left(\left|\sum_{k=1}^{d}\xi_{1,k}\xi_{2,k}\right|^{\alpha}\ind_{\{\|X_1^{(d)}\|_2^2\geq s a(n),\|X_2^{(d)}\|_2^2\geq s a(n)\}}\right)\\
&=\frac nd \cdot \frac{1}{s^{3\alpha}\varepsilon^{\alpha}}\E \left(\left|\sum_{k=1}^{d}\xi_{1,k}\xi_{2,k}\right|^{\alpha}\ind_{\{\|X_1^{(d)}\|_2^2\geq s a(n),\|X_2^{(d)}\|_2^2\geq s a(n)\}}\right)
,\quad d\to\infty,
\end{align*}
where for the asymptotic  equivalence we used~\eqref{eq:heyde_ex3}. 
Using the subadditivity of $x\mapsto x^{\alpha}$, $\alpha\in (0,1)$, we conclude that
\begin{multline*}
\E \left(\left|\sum_{k=1}^{d}\xi_{1,k}\xi_{2,k}\right|^{\alpha}\ind_{\{\|X_1^{(d)}\|_2^2\geq s a(n),\|X_2^{(d)}\|_2^2\geq s a(n)\}}\right)\leq \sum_{k=1}^{d}\E |\xi_{1,k}|^{\alpha}\E|\xi_{2,k}|^{\alpha}\ind_{\{\|X_1^{(d)}\|_2^2\geq s a(n),\|X_2^{(d)}\|_2^2\geq s a(n)\}}\\
=d \left(\E |\xi_{1,1}|^{\alpha} \ind_{\{\|X_1^{(d)}\|_2^2\geq s a(n)\}}\right)^2=d \left(\E |\xi_{1,1}|^{\alpha} \ind_{\{\xi_{1,1}^2+\cdots+\xi_{1,d}^2\geq d^{1/\alpha}s a(n)\}}\right)^2.
\end{multline*}
Further, since $(a+b)^2\leq 2a^2+2b^2$, we have
\begin{align*}
&\hspace{-0.5cm}d \left(\E |\xi_{1,1}|^{\alpha} \ind_{\{\xi_{1,1}^2+\cdots+\xi_{1,d}^2\geq d^{1/\alpha}s a(n)\}}\right)^2\\
&\leq 2d \left(\E |\xi_{1,1}|^{\alpha} \ind_{\{\xi_{1,1}^2 \geq d^{1/\alpha}s a(n)/2\}}\right)^2+2d \left(\E |\xi_{1,1}|^{\alpha} \ind_{\{\xi_{1,2}^2+\cdots+\xi_{1,d}^2 \geq d^{1/\alpha}s a(n)/2\}}\right)^2\\
&=  2d \left(\E |\xi_{1,1}|^{\alpha} \ind_{\{\xi_{1,1}^2 \geq d^{1/\alpha} sa(n)/2\}}\right)^2+2d \left(\E |\xi_{1,1}|^{\alpha} \mmp\{\xi_{1,2}^2+\cdots+\xi_{1,d}^2 \geq d^{1/\alpha}s a(n)/2\}\right)^2.
\end{align*}
By Eq.~(4) in~\cite{Heyde:1968},
\begin{equation*}
   \mmp\{\xi_{1,2}^2+\cdots+\xi_{1,d}^2 \geq d^{1/\alpha}s a(n)/2\} \sim  (d-1) \cdot \mmp\{\xi^2 \geq d^{1/\alpha}s a(n)/2\}
\sim (s/2)^{-\alpha}/n. 
\end{equation*}
It remains to check that 
\begin{equation*}
   \lim_{d\to\infty} n\left(\E |\xi|^{\alpha} \ind_{\{\xi^2\geq d^{1/\alpha} s a(n)/2\}}\right)^2=0.
\end{equation*}
To this end, it is clearly sufficient to show that
\begin{equation}\label{example3_check}
   \lim_{A\to\infty}\lim_{n\to\infty} n \left(\E |\xi|^{\alpha} \ind_{\{\xi^2\geq A a(n)\}}\right)^2=0.
\end{equation}
This can be accomplished by an appeal to formula (5.21) on p.~579 in~\cite{Feller_Vol2} applied with $\beta=\alpha/2$. According to this formula,
\begin{equation*}
   \E |\xi|^{\alpha} \ind_{\{\xi^2\geq A a(n)\}}=\E (\xi^2)^{\alpha/2} \ind_{\{\xi^2\geq A a(n)\}}~\sim~\frac{4-2\alpha}{\alpha}A^{\alpha/2}(a(n))^{\alpha/2}\mmp\{\xi^2\geq A a(n)\},\quad n\to\infty.
\end{equation*}
Thus,
\begin{equation*}
   n \left(\E |\xi|^{\alpha} \ind_{\{\xi^2\geq A a(n)\}}\right)^2 
\to
\left(\frac{4-2\alpha}{\alpha}\right)^2 A^{-\alpha},\quad n\to\infty,
\end{equation*}
and~\eqref{example3_check} follows.
\end{proof}

\section*{Acknowledgments}
AM gratefully acknowledges the financial support and hospitality of the University of Angers during his stay in December 2022-March 2023. ZK was supported by the German Research Foundation under Germany's Excellence Strategy  EXC 2044 -- 390685587, Mathematics M\"unster: Dynamics - Geometry - Structure and  by the DFG priority program SPP 2265 Random Geometric Systems. This project has received funding from the European Research Council (ERC) under the European Union's Horizon 2020 research and innovation programme under the Grant Agreement No.\ 759702 and from Centre Henri Lebesgue, programme ANR-11-LABX-0020-0.

\bibliographystyle{plainnat}
\bibliography{Crinkled}

\end{document}